\documentclass{article}
\usepackage{amssymb,latexsym,amsmath,graphicx,amsthm,hyperref,tipa,comment,float}

\newtheorem{fivetotheonethird}{Proposition}
\newtheorem{phit}[fivetotheonethird]{Proposition}
\newtheorem{powersoftwo}[fivetotheonethird]{Proposition}
\newtheorem{largerthantwo}[fivetotheonethird]{Proposition}
\newtheorem{smallerthanphi}[fivetotheonethird]{Proposition}
\newtheorem{onepointfive}[fivetotheonethird]{Proposition}
\newtheorem{adhocone}[fivetotheonethird]{Proposition}
\newtheorem{adhoctwo}[fivetotheonethird]{Proposition}

\newtheorem{infareasquared}[fivetotheonethird]{Proposition}

\newtheorem{folk}{Lemma}
\newtheorem{alphalargerthangold}[folk]{Lemma}
\newtheorem{eventuallycomplete}[folk]{Lemma}
\newtheorem{boring}[folk]{Lemma}
\newtheorem{lessthantwice}[folk]{Lemma}
\newtheorem{ezupper}[folk]{Lemma}
\newtheorem{ezinclusion}[folk]{Lemma}

\newtheorem{corro}{Corollary}

\begin{document}

\vspace*{-2cm}

\Large
 \begin{center}
Completeness of exponentially increasing sequences  \\ 

\hspace{10pt}

\large
Wouter van Doorn \\

\hspace{10pt}

\end{center}

\hspace{10pt}

\normalsize

\vspace{-25pt}
\centerline{\bf Abstract}
For fixed positive reals $t$ and $\alpha$, consider the sequence $S_t(\alpha) = (s_1, s_2, \ldots, )$ with $s_n = \left \lfloor t\alpha^n \right \rfloor$. In 1964, Graham managed to characterize those pairs $(t, \alpha)$ with $0 < t < 1$ and $1 < \alpha < 2$ for which every large enough integer can be written as the sum of distinct elements of $S_t(\alpha)$. We show that his methods can be applied to deal with many other pairs of $(t, \alpha)$ as well.

\section{Introduction}
For a sequence or multiset $S$ of positive integers we define $P(S)$ as the set of all integers that can be written as the sum of distinct elements of $S$. We say that $S$ is complete if $\mathbb{N} \setminus P(S)$ is finite, and we say that $S$ is entirely complete if $P(S) = \mathbb{N}$. For positive real numbers $t$ and $\alpha$, in this paper we are interested in the completeness of the sequence $S_t(\alpha) = (s_1, s_2, \ldots, )$, where $s_n = \left \lfloor t\alpha^n \right \rfloor$. \\

Research into this question started with \cite{gra}, where Graham first showed that for $1 < \alpha \le 5^{1/3}$, the sequence $S_t(\alpha)$ is entirely complete if $t < 1$. He then more generally determined the full set of pairs $(t, \alpha)$ with $t < 1$ and $1 < \alpha < 2$ for which $S_t(\alpha)$ is (entirely) complete. Interestingly, for the $t$ and $\alpha$ in this square, $S_t(\alpha)$ is complete if, and only if, it is entirely complete. \\

With the aforementioned square having been dealt with, Graham \cite{gra1} asked in full generality for which values of $t > 0$ and $\alpha > 0$ the sequence $S_t(\alpha)$ is complete. This question was repeated by Erd\H{o}s and Graham in \cite{cnt}, and it is now listed as Problem \#349 at \cite{bloom349}. The goal of this paper is to revisit \cite{gra} and see how much further its ideas can be pushed. \\

We will generalize the result from Graham's paper that we just mentioned, and show that if $1 < \alpha \le 5^{1/3}$, then $S_t(\alpha)$ is entirely complete if, and only if, $$t < \min\left(\frac{2}{\alpha}, \frac{3}{\alpha^2}, \frac{5}{\alpha^{3}}\right).$$ On the other hand, if \begin{equation*}\alpha \ge \varphi := \frac{1 + \sqrt{5}}{2} \text{ and } t \ge \max\left(\min\Big(\frac{3}{\alpha^2}, \frac{5}{\alpha^{3}}\Big), 1\right), \end{equation*} then $S_t(\alpha)$ is not complete. Combined with Graham's results, this lets us finish off the case $\alpha \ge \varphi$. \\

As for $1 < \alpha < \varphi$, it is plausible that $S_t(\alpha)$ is complete for all $t > 0$, but this remains open in general. On the other hand, for any $T$ and any $\epsilon > 0$, in our final section we show how one can in principle prove this conjecture for all (uncountably many) pairs $(t, \alpha)$ with $t < T$ and $\alpha < \varphi - \epsilon$ at once, by a finite computation. We then give two examples of what such proofs might look like (one that is human-readable, and one with the help of a computer) and show the completeness of $S_t(\alpha)$ for various intervals of $t$ and $\alpha$. Finally, an idea used in these proofs is then applied to furthermore show that $S_t(\alpha)$ is complete for all $\alpha < 1 + \frac{1}{\lceil t \rceil + 2\lceil \sqrt{t} \rceil}$.

\section{Preliminary lemmas}
Before we fully dive in, in this section we mention a few small general lemmas that we will repeatedly make use of. Most of them are very elementary and well-known, but we will record them anyway. We start off with a lemma that Graham \cite{gra} attributes to Folkman, although we have not been able to find a reference. For completeness' sake (no pun intended) we also provide a proof.

\begin{folk} \label{folk}
Assume that a positive integer $m \notin P(S_t(\alpha))$ and a non-negative integer $r$ exist with $$s_1 + \cdots + s_r < m < s_{r+2},$$ while $s_n + s_{n+1} \le s_{n+2}$ holds for all $n > r$. Then $$m + s_{r+3} + s_{r+5} + \cdots + s_{r + 2k+1} \notin P(S_t(\alpha))$$ for all $k \ge 1$.
\end{folk}

\begin{proof}
For an integer $k \ge 1$ define 

\begin{equation*}
m_k = m + \sum_{i=1}^{k-1} s_{r + 2i + 1},
\end{equation*}

and assume by induction $m_k \notin P(A)$ while

\begin{equation*}
\sum_{i=1}^{r+2(k-1)} s_i < m_k < s_{r+2k}.
\end{equation*}

This is certainly true for $k = 1$ by assumption. Adding $s_{r+2k+1}$ to these inequalities and using $s_{n} + s_{n+1} \le s_{n+2}$ on both the left- and the right-hand sides, we deduce

\begin{equation*}
\sum_{i=1}^{r+2k} s_i < m_{k+1} < s_{r+2k+2}.
\end{equation*}

We therefore see that, if $m_{k+1} \in P(S_t(\alpha))$, then you would need to use $s_{r+2k+1}$ in any representation of $m_{k+1}$. But this would imply that $m_k = m_{k+1} - s_{r+2k+1}$ has a representation as well, which contradicts the induction hypothesis.
\end{proof}

Lemma \ref{folk} comes in handy whenever $\alpha \ge \varphi$, as we then have the following inequalities, the first part of which is Lemma $3$ in \cite{gra}.

\begin{alphalargerthangold} \label{alphalarge}
If $\alpha \ge \varphi$, then $s_{n} + s_{n+1} \le s_{n+2}$ for all $n \in \mathbb{N}$. And if $\alpha > 2$, then $1 + s_1 + \cdots + s_n < s_{n+1}$ if $n$ is large enough.
\end{alphalargerthangold}

\begin{proof}
If $\alpha \ge \varphi$, we note that 
\begin{align*}
t \alpha^n + t \alpha^{n+1} &= t\left(\frac{1}{\alpha^2} + \frac{1}{\alpha}\right)\alpha^{n+2} \\
&\le t\left(\frac{1}{\varphi^2} + \frac{1}{\varphi}\right)\alpha^{n+2} \\
&= t\alpha^{n+2}.
\end{align*}

Taking the floor on both sides and realizing that the inequality $$\lfloor t \alpha^n \rfloor + \lfloor t \alpha^{n+1}\rfloor  \le \lfloor t \alpha^n + t \alpha^{n+1} \rfloor$$ holds in full generality, finishes the proof for the first part of the lemma. \\

If $\alpha > 2$, let $n$ be large enough so that $$t\alpha^{n+1}(\alpha - 2) > 2\alpha - 2.$$ Then $$\frac{t\alpha^{n+1}}{\alpha - 1} < t\alpha^{n+1} - 2,$$ as $$(t\alpha^{n+1} - 2)(\alpha - 1) - t\alpha^{n+1} = t\alpha^{n+1}(\alpha - 2) - 2\alpha + 2 > 0.$$ We therefore get
\begin{align*}
1 + \sum_{i=1}^{n} s_i&\le 1 + \sum_{i=1}^{n} t\alpha^i \\
&= 1 + \frac{t(\alpha^{n+1} - \alpha)}{\alpha - 1} \\
&< 1 + \frac{t\alpha^{n+1}}{\alpha - 1} \\
&< t\alpha^{n+1} - 1 \\
&< s_{n+1}. \qedhere
\end{align*}
\end{proof}

Combining Lemma \ref{folk} and the first part of Lemma \ref{alphalarge} provides the following corollary.

\begin{corro} \label{corro}
If $\alpha \ge \varphi$ and $s_1 + \cdots + s_r < m < s_{r+1}$ for some $m \ge 1$, $r \ge 0$, then $S_t(\alpha)$ is not complete.
\end{corro}

In the next section we will apply Corollary \ref{corro} to show that certain pairs of $(\alpha, t)$ lead to sequences which are not complete. But first we will quickly mention one more general lemma and two lemmas specific to our sequence, all geared towards showing that some sequences are complete.

\begin{lessthantwice}[\cite{gra}, Lemma $2$] \label{lessthantwice}
Assume a non-negative integer $r$ exists such that the inequality $s_{n+1} \le 1 + s_1 + \cdots + s_n$ holds for all non-negative $n \le r$, while $s_{n+1} \le 2s_n$ holds for all $n > r$. Then $S_t(\alpha)$ is entirely complete.
\end{lessthantwice}

\begin{boring} \label{bore}
Assume $t \ge 1$. If $1 < \alpha < \frac{3}{2}$, then the inequality $s_{n+1} \le 2s_n$ holds for all $n \ge 1$. If $\frac{3}{2} \le \alpha < \varphi$, then $s_{n+1} \le 2s_n$ holds for all $n \ge 2$. And if $\varphi \le \alpha < 5^{1/3}$, then $s_{n+1} \le 2s_n$ holds for all $n \ge 3$.
\end{boring}

\begin{proof}
One can check that these claims follow from the first part of \cite[Lemma $4$]{gra}.
\end{proof}

\begin{eventuallycomplete} \label{eventually}
Let $t \ge 1$ and $1 < \alpha < \varphi$, and assume that positive integers $r$ and $X$ exist such that $m \in P(\{s_1, \ldots, s_r \})$ for all $m$ with $X \le m < X + s_{r+1}$. Then $S_t(\alpha)$ is complete.
\end{eventuallycomplete}

\begin{proof}
The hypothesis implies in particular that $m \in P(\{s_1, \ldots, s_r, s_{r+1} \})$ for all $m$ with $X + s_{r+1} \le m < X + 2s_{r+1}$. Since $s_{r+2} \le 2s_{r+1}$ by Lemma \ref{bore}, we deduce $m \in P(\{s_1, \ldots, s_r, s_{r+1} \})$ for all $m$ with $X \le m < X + s_{r+2}$, so we are done by induction, and see that all $m \ge X$ are representable.
\end{proof}

\section{Main results}
We are now ready to deal with the completeness of $S_t(\alpha)$. And to provide the full picture, we swiftly deal with the near-trivial case where $\alpha$ does not belong to the interval $(1, 2)$ first, before moving on to the more interesting cases.

\begin{largerthantwo}
If $\alpha \notin [1, 2]$, then $S_t(\alpha)$ is not complete for any $t > 0$. And if $\alpha = 1$, then $S_t(\alpha)$ is (entirely) complete if, and only if, $t \in [1, 2)$.
\end{largerthantwo}

\begin{proof}
If $\alpha < 1$, then $s_n = 0$ for all large $n$, so that $P(S_t(\alpha))$ is actually finite. If $\alpha = 1$, then $s_n = \left \lfloor t \right \rfloor$ for all $n$, so that $P(S_t(\alpha)) = \left \lfloor t \right \rfloor \mathbb{N}$. Finally, if $\alpha > 2$, then $s_n - 1 \notin P(S_t(\alpha))$ for all large enough $n$, by the second part of Lemma \ref{alphalarge}. 
\end{proof}

\begin{powersoftwo}
If $\alpha = 2$, then $S_t(\alpha)$ is (entirely) complete if, and only if, $t = \frac{1}{2^k}$ for some $k \ge 1$.
\end{powersoftwo}

\begin{proof}
For $t = \frac{1}{2^k}$ the completeness of $S_t(\alpha)$ follows from the completeness of the sequence of powers of two, while for $t \ge 1$ we can apply Corollary \ref{corro} with $r = 0, m = 1$ to deduce that $S_t(\alpha)$ is not complete. Finally, if $t < 1$ and $t \neq \frac{1}{2^k}$, let $n$ be the smallest index with $s_n \ge 1$. We then see $s_n = 1$, so write $t2^n = 1 + \epsilon$ for some $\epsilon \in (0, 1)$ and define $j = \left \lceil \frac{-\log \epsilon}{\log 2} \right \rceil$. We then have $s_{n+i} = 2^i$ for all $i$ with $0 \le i < j$ and $s_{n+j} = 2^j + 1$, so that we can apply Corollary \ref{corro} with $r = n+j-1, m = 2^j$.
\end{proof}

With these trivial cases out of the way, we may from now on assume $1 < \alpha < 2$, while we may further assume $t \ge 1$ by the results in \cite{gra}. We will split up the interval $1 < \alpha < 2$ into four different regions; $1 < \alpha < \frac{3}{2}$, $\frac{3}{2} \le \alpha < \varphi$, $\varphi \le \alpha < 5^{1/3}$ and $5^{1/3} \le \alpha < 2$. We will deal with them in order from large to small.

\begin{fivetotheonethird} \label{five}
If $5^{1/3} \le \alpha < 2$, there is no value of $t \ge 1$ for which $S_t(\alpha)$ is complete.
\end{fivetotheonethird}

\begin{proof}
If $t \ge \frac{2}{\alpha}$ we see $s_1 \ge 2$, which means we can apply Corollary \ref{corro} with $m = 1$, $r = 0$ to deduce that $S_t(\alpha)$ is not complete. So we may assume that $t$ is smaller than $\frac{2}{\alpha}$, implying $s_1 \le 1$. Now, if $t \ge \frac{3}{\alpha^2}$, then $s_2 \ge 3$, in which case we apply Corollary \ref{corro} with $m = 2$, $r = 1$. So we may further assume $t < \frac{3}{\alpha^2}$, which gives $s_2 \le 2$. But we then have $t \ge 1 \ge \frac{5}{\alpha^3}$, implying $s_3 \ge 5$. And in this case we apply Corollary \ref{corro} with $m = 4$, $r = 2$.
\end{proof}

\begin{phit} \label{phit}
If $\varphi \le \alpha < 5^{1/3}$, then $S_t(\alpha)$ is (entirely) complete if, and only if, $t < \min(\frac{3}{\alpha^2}, \frac{5}{\alpha^{3}})$.
\end{phit}

\begin{proof}
For $t \ge \min(\frac{3}{\alpha^2}, \frac{5}{\alpha^{3}})$ we can apply the same proof that we just used for Proposition \ref{five} to conclude that $S_t(\alpha)$ is not complete. So it remains to show the converse; prove the completeness of $S_t(\alpha)$ if $t < \min(\frac{3}{\alpha^2}, \frac{5}{\alpha^{3}})$, and by Theorem $2$ in \cite{gra} we may assume $t \ge 1$. \\

By the assumed bounds on $\alpha$ and $t$ we claim $s_1 = 1, s_2 = 2, s_3 = 4$. Indeed,
\begin{align*}
1 &< \alpha \le t\alpha < \frac{3}{\alpha} < 2, \\
2 &< \alpha^2 \le t\alpha^2 < 3, \\
4 &< \alpha^3 \le t\alpha^3 < 5.
\end{align*}

Since $s_{n+1} \le 2s_n$ for all $n \ge 3$ by Lemma \ref{bore}, Lemma \ref{lessthantwice} implies that $S_t(\alpha)$ is entirely complete.
\end{proof}

We remark that the case $\alpha \ge \varphi$ is now fully dealt with. We therefore only need to consider those pairs $(t, \alpha)$ with $t \ge 1$ and $1 < \alpha < \varphi$, and it is thought (see e.g. \cite[p. 57]{cnt}) that $S_t(\alpha)$ is complete for all these values. We will not be able to fully settle this problem, but we will at least determine all those pairs for which $S_t(\alpha)$ is entirely complete. For example, for all $\alpha$ with $1 < \alpha < \varphi$, we will see that the resulting sequence is entirely complete for all $t \le \frac{9 - 3\sqrt{5}}{2} \approx 1.15$.

\begin{onepointfive} \label{threethirds}
If $\frac{3}{2} \le \alpha < \varphi$, then $S_t(\alpha)$ is entirely complete if, and only if, $t < \frac{3}{\alpha^2}$. In particular, if $t \le \frac{9 - 3\sqrt{5}}{2}$, then $S_t(\alpha)$ is entirely complete for all these values of $\alpha$.
\end{onepointfive}

\begin{proof}
Analogously to the second part of the proof of Proposition \ref{phit}, by $t \ge 1$ and the given bounds on $\alpha$ we once again deduce $s_1 = 1, s_2 = 2$. And by applying both Lemma \ref{bore} and Lemma \ref{lessthantwice} once more, $S_t(\alpha)$ is entirely complete. Conversely, if $t \ge \frac{3}{\alpha^2}$, then $s_2 \ge 3$, so that $\{1, 2\} \not \subset P(S_t(\alpha))$.
\end{proof}

\begin{smallerthanphi} \label{small}
If $1 < \alpha < \frac{3}{2}$, then $S_t(\alpha)$ is entirely complete if, and only if, $t < \frac{2}{\alpha}$. In particular, if $t \le \frac{4}{3}$, then $S_t(\alpha)$ is entirely complete for all these values of $\alpha$.
\end{smallerthanphi}

\begin{proof}
If $t \ge \frac{2}{\alpha}$, then $1 \notin P(S_t(\alpha))$, so that $S_t(\alpha)$ is not entirely complete. On the other hand, if $1 \le t < \frac{2}{\alpha}$, then we deduce $s_1 = 1$, while $s_{n+1} \le 2s_n$ for all $n \ge 1$ by Lemma \ref{bore}. Once more, we are done by Lemma \ref{lessthantwice}.
\end{proof}

The only remaining case left is when $1 < \alpha \le \varphi$ and $t \ge \min\left(\frac{2}{\alpha}, \frac{3}{\alpha^2} \right)$. This is furthermore the only case where $S_t(\alpha)$ can be complete without being entirely complete. And this brings us to a new section.

\section{Computational possibilities}
In this final section we would like to give a taste of where this research could be taken next. For this, define $U \subset \mathbb{R}^2$ to be the set of all pairs $(t, \alpha)$ with $1 < \alpha < \varphi$ and $t \ge \min\left(\frac{2}{\alpha}, \frac{3}{\alpha^2} \right)$. By the results in the previous section we may assume $(t, \alpha) \in U$, in which case $S_t(\alpha)$ is necessarily not entirely complete, as either $s_1 \ge 2$, or $s_1 = 1$ and $s_2 \ge 3$. However, for any $(t, \alpha) \in U$ it is sufficient to provide $r$ and $X$ for which one can apply Lemma \ref{eventually} to still conclude that $S_t(\alpha)$ is complete. Moreover, if such $r$ and $X$ are found, let, for $1 \le i \le r+1$, $\epsilon_i > 0$ be such that $(t + \epsilon_i)(\alpha + \epsilon_i)^i = s_i + 1$, and define $\epsilon = \min_i \epsilon_i$. We then get that for every $\delta_1$ and $\delta_2$ smaller than $\epsilon$, the first $r+1$ elemens of $S_t(\alpha)$ coincide with the first $r+1$ elements of $S_{t+\delta_1}(\alpha + \delta_2)$, implying that $S_{t'}(\alpha')$ is complete for all $(t', \alpha')$ with $t \le t' < t + \epsilon$ and $\alpha \le \alpha' < \alpha + \epsilon$. For any compact set $C \subset U$ it is in principle possible to tile the entirety of $C$ with these kinds of squares (and furthermore mathematically prove that you have done so), and with a computer one can automate this search. \\

What we will do however, is a related but slightly different approach, where we work in the other direction. We start off with a given set $C \subset U$, and then note that there are only finitely many distinct possibilities for the values of, say, $s_1, s_2, \ldots, s_{10}$. All of these possible values can in principle be enumerated, and if it turns out that every single possibility leads to the existence of $X$ and $r$ with which to apply Lemma \ref{eventually}, then $S_t(\alpha)$ is complete for all $(t, \alpha) \in C$. First off we will prove such a result in a way that can easily by checked by hand, after which we will show what automating this process looks like.

\begin{adhocone} \label{adhocuno}
If $1 < \alpha \le \frac{5}{4}$, then $S_t(\alpha)$ is complete for all $t < \frac{4}{\alpha}$.
\end{adhocone}

Before we start the proof, let us quickly mention and prove two small, but valuable lemmas.

\begin{ezupper} \label{ez}
For all $\alpha$ and $t$ we have $s_{n} < \alpha(s_{n-1} + 1)$ for all $n \ge 2$.
\end{ezupper}

\begin{proof}
An easy calculation;
\begin{align*}
s_{n} &= \left \lfloor t\alpha^{n} \right \rfloor \\
&\le \alpha(t\alpha^{n-1}) \\
&< \alpha(s_{n-1} + 1). \qedhere
\end{align*}
\end{proof}

\begin{ezinclusion} \label{incl}
If $1 < \alpha \le 1 + \frac{1}{x}$ for some $x > t$, then $m \in S_t(\alpha)$ for all $m$ with $s_1 \le m \le x$.
\end{ezinclusion}

\begin{proof}
Without loss of generality assume $s_1 < m \le x$, and let $n \ge 2$ be the largest integer with $s_{n-1} < m$. By Lemma \ref{ez} we then have 
\begin{align*}
m &\le s_n \\
&< \alpha(s_{n-1} + 1) \\
&\le \alpha m \\
&\le m + \frac{m}{x} \\
&\le m+1.
\end{align*} 

Since $m \le s_n < m+1$ we indeed get $s_n = m$.
\end{proof}

\begin{proof}[Proof of Proposition \ref{adhocuno}]
As a first remark, by applying Lemma \ref{incl} with $x = 4$, we deduce $\{3, 4\} \subset S_t(\alpha)$. \\

Now, if $\alpha \le \frac{7}{6}$, then we even have $\{3, 4, 5, 6\} \subset S_t(\alpha)$, by applying Lemma \ref{incl} with $x = 6$ instead. And by employing Lemma \ref{ez} with $n$ the largest index with $s_{n-1} < 7$, we moreover get $s_n \in \{7, 8\}$. As $P(\{3, 4, 5, 6 \})$ contains $\{3, 4, \ldots, 11\}$, we can finish this case by applying Lemma \ref{eventually} with $X = 3, r = n-1$. \\

We will therefore assume $\frac{7}{6} < \alpha \le \frac{5}{4}$ from now on, and we may further assume $\{5, 6\} \not \subset S_t(\alpha)$, as we would otherwise be done by the previous argument. \\

By Lemma \ref{ez} and the assumption $\alpha \le \frac{5}{4}$, we note that $s_n \in \{5, 6\}$ if $s_{n-1}$ is the largest element smaller than $5$. Depending on the value of $s_n$ we now have to distinguish between two cases, and let us first assume $s_n = 6$. For $s_{n+1}$ we then, on the one hand, get 
\begin{align*}
t\alpha^{n+1} &> \frac{7}{6} t\alpha^{n} \\
&\ge \frac{7}{6} s_{n} \\
&= 7.
\end{align*}

While on the other hand,
\begin{align*}
t\alpha^{n+1} &\le \left(\frac{5}{4}\right)^2t\alpha^{n-1} \\
&< \frac{125}{16} \\
&< 8. 
\end{align*}

And we conclude $s_{n+1} = 7$. For $s_{n+2}$ we analogously have $8 \le s_{n+2} \le 9$, while $s_{n+3} \le 12$. \\

Regardless of whether $\{3, 4, 6, 7, 8\}$ or $\{3, 4, 6, 7, 9\}$ is contained in $S_t(\alpha)$, in both cases one can check that we have $\{9, 10, \ldots, 20\} \subset P(\{s_1, \ldots, s_{n+2}\})$. This in turn implies that we are done by applying Lemma \ref{eventually} with $X = 9, r = n+2$. \\

The second case we have to deal with is $s_n = 5$ with $6 \notin S_t(\alpha)$. As it turns out however, this case works analogously, and we get that either $S_1 = \{3, 4, 5, 7, 8\}$ or $S_2 = \{3, 4, 5, 7, 9\}$ is contained in $S_t(\alpha)$, while $S_t(\alpha)$ also contains either $10, 11$ or $12$. With both $S_1$ and $S_2$ we are again lucky enough to get $\{9, 10, \ldots, 20\} \subset P(S_i)$, so one final application of Lemma \ref{eventually} finishes the proof.
\end{proof}

Even though it was not the most thrilling proof of all time, let us quickly recap it anyway, to see how one could perhaps automate it and extend Proposition \ref{adhocuno} to larger intervals of $t$ and $\alpha$. By repeated applications of Lemma \ref{ez} we deduced that one of the following sets must be contained in $S_t(\alpha)$ (where $7 \le x \le 8$ and $9 \le y \le 12$):
\begin{align*}
&\{3, 4, 5, 6, x \}, \\
&\{3, 4, 5, 7, 8, y \}, \\
&\{3, 4, 5, 7, 9, y \}, \\
&\{3, 4, 6, 7, 8, y \}, \text{or} \\
&\{3, 4, 6, 7, 9, y \}.
\end{align*}

And in every case, the elements excluding the last one generate a large enough interval to apply Lemma \ref{eventually} with. To give a visual representation of what this looks like, here is a plot of which values of $t$ and $\alpha$ lead to which of the above subsets.

\begin{figure}[H]
    \centering
    \includegraphics[width=0.8\textwidth]{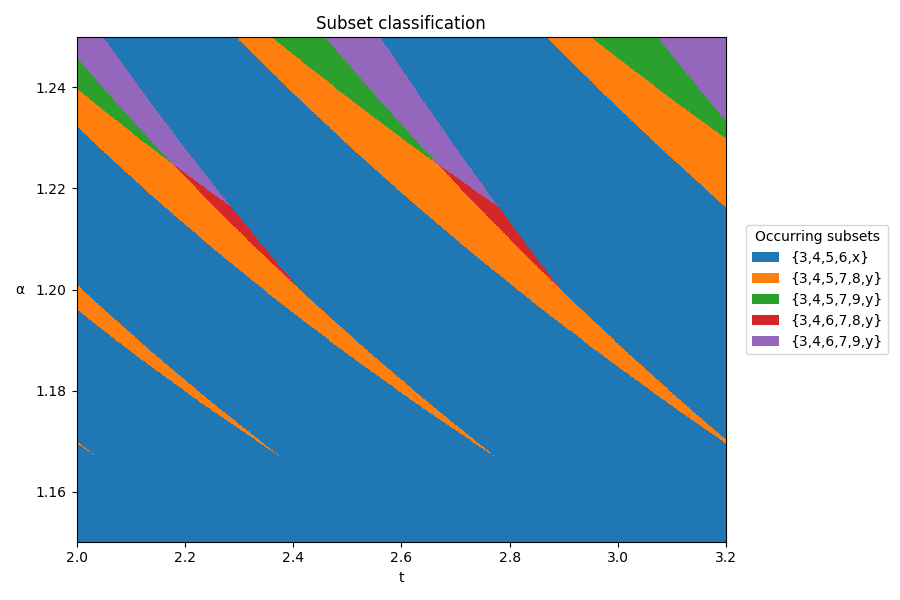}
    \label{fig1}
\end{figure}

Now, if we were to extend the range of $t$ or $\alpha$, the potential number of subsets one has to check increases quite rapidly of course, but a computer seems very amenable to deal with such issues. We have therefore done some coding in order to extend Proposition \ref{adhocuno} even further, which leads to the following computer-assisted result.

\begin{adhoctwo} \label{adhocduo}
If $1.3 < \alpha \le 1.4$, then $S_t(\alpha)$ is complete for all $t \le 3$. \\
If $1.2 < \alpha \le 1.3$, then $S_t(\alpha)$ is complete for all $t \le 5$. \\
If $1.1 < \alpha \le 1.2$, then $S_t(\alpha)$ is complete for all $t \le 10$. \\
If $1 < \alpha \le 1.1$, then $S_t(\alpha)$ is complete for all $t \le 50$. 
\end{adhoctwo}

\begin{proof}
Given any rectangle $C \subset U$, let $(t_1, \alpha_1)$ and $(t_2, \alpha_2)$ be the bottom-left vertex and top-right vertex with sequences $(s_1, s_2, \ldots)$ and $(s'_1, s'_2, \ldots)$ respectively. Let $i_1 < i_2 < \cdots < i_k$ be those indices for which $s_{i_j} = s'_{i_j}$ (which we call the overlap) and let $i_0$ be the first index with $s_{i_0} \neq s'_{i_0}$ (which we call the dummy), and note that $i_0 < i_k$ and $i_0 > i_k$ are both possible. We then deduce that all pairs $(t, \alpha) \in C$ with sequence $(s''_1, s''_2, \ldots)$ have this overlap as well (that is, $s_{i_j} = s'_{i_j} = s''_{i_j}$ for all $j$ with $1 \le j \le k$), while $s_{i_0} \le s''_{i_0} \le s'_{i_0}$. In particular, if for all $s''_{i_0}$ with $s_{i_0} \le s''_{i_0} \le s'_{i_0}$ there exists an $X$ such that $m \in P(\{s_{i_1}, s_{i_2}, \ldots, s_{i_{k-1}}, \min(s''_{i_0}, s_{i_k}) \})$ for all $m$ with $X \le m < X + \max(s''_{i_0}, s_{i_k})$, then Lemma \ref{eventually} applies for all $(t, \alpha) \in C$. \\

On the other hand, if for some of these sequences no such $X$ can be found, then we can split the rectangle $C$ into $4$ smaller rectangles, and try again. This is precisely what we have coded, and the results can be found on the author's GitHub page\footnote{See https://github.com/Woett/Complete-sequences-data.}. There, one can find specific partitions (of the different regions of $U$ mentioned in Proposition \ref{adhocduo}) into rectangles for which Lemma \ref{eventually} applies to the overlap and dummy, as explained above. We note that, for $\alpha \le 1.2$ we may assume $t \ge 2$ by Proposition \ref{adhocuno}, while for $\alpha \le 1.1$ it is sufficient to take $\alpha \ge 1.015$ by Proposition \ref{tothemoontwo} below.
\end{proof}

If one wants to verify the data from the proof of Proposition \ref{adhocduo}, this verification process essentially consists of the following three parts:

\begin{enumerate}
	\item Check that the rectangles together cover the entire region.
	\item Calculate the overlap and dummy of the bottom-left and top-right vertices of every rectangle.
	\item Check that, with this overlap and dummy, Lemma \ref{eventually} may indeed be invoked.
\end{enumerate}

All in all, combining our results with those from \cite{gra} now provides the following picture for $t \le 2$.

\begin{figure}[H]
    \centering
    \includegraphics[width=0.8\textwidth]{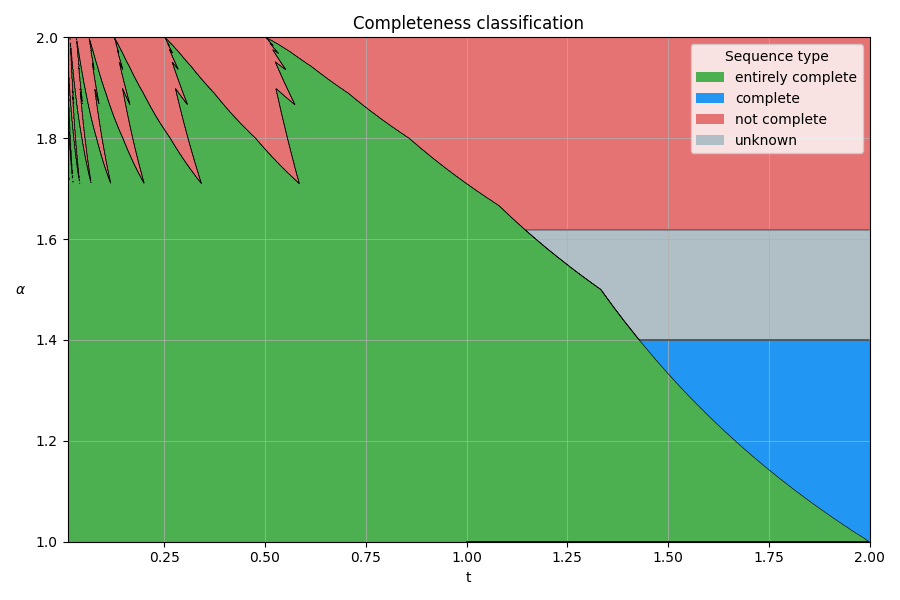}
    \label{fig2}
\end{figure}

For a final note, we would like to mention one nice application of Lemma \ref{incl}. If one plots $(t, \alpha)$ in the plane as in the above picture, then up till now, the region for which we have managed to show that $S_t(\alpha)$ is complete, is bounded. So to finish this paper, let us prove it for a region with infinite area instead.

\begin{infareasquared} \label{tothemoontwo}
The sequence $S_t(\alpha)$ is complete for all $\alpha$ with $$1 < \alpha \le 1 + \frac{1}{\lceil t \rceil + 2\lceil \sqrt{t} \rceil}.$$
\end{infareasquared}

\begin{proof}
For ease of notation, define $v = \lceil t \rceil$ and $w = \lceil \sqrt{t} \rceil$. By noting that $s_1 \le v$, we get $A := \{v, v + 1, \ldots, v + 2w\} \subset S_t(\alpha)$ by Lemma \ref{incl}. And with $r$ such that $s_r = v + 2w$, we further obtain $s_{r+1} \le v + 2w + 2$ by Lemma \ref{ez}. \\

Now, by defining $X = vw + \frac{1}{2}w(w-1)$ and taking sums of $w$ elements from $A$, we then deduce $m \in P(\{s_1, \ldots, s_r \})$ for all $m$ with $X \le m \le X + w + w^2$. And by taking sums of $w+1$ elements from $A$, we furthermore deduce that this also holds for all $m$ with $X + v + w \le m \le X + v + 2w + w^2$. In particular, combining the intervals we get $m \in P(\{s_1, \ldots, s_r \})$ for all $m$ with $X \le m < X + s_{r+1}$. Hence, the proof is finished by applying Lemma \ref{eventually}, as we are free to assume $t \ge 1$.
\end{proof}

\end{document}